\documentclass{article}
\usepackage{amsfonts}
\usepackage{amssymb}
\usepackage{graphicx}
\usepackage{amsmath}

\newtheorem{theorem}{Theorem}

\newtheorem{corollary}[theorem]{Corollary}

\newtheorem{proposition}[theorem]{Proposition}
\newtheorem{remark}[theorem]{Remark}

\newenvironment{proof}[1][Proof]{\textbf{#1.} }{\ \rule{0.5em}{0.5em}}
\input{tcilatex}

\begin{document}

\title{Sums of powers of Fibonacci and Lucas polynomials in terms of
Fibopolynomials}
\author{Claudio de J. Pita Ruiz V. \\
Universidad Panamericana\\
Mexico City, Mexico\\
email: cpita@up.edu.mx}
\date{}
\maketitle

\begin{abstract}
We consider sums of powers of Fibonacci and Lucas polynomials of the form $%
\sum_{n=0}^{q}F_{tsn}^{k}\left( x\right) $ and $\sum_{n=0}^{q}L_{tsn}^{k}%
\left( x\right) $, where $s,t,k$ are given natural numbers, together with
the corresponding alternating sums $\sum_{n=0}^{q}\left( -1\right)
^{n}F_{tsn}^{k}\left( x\right) $ and $\sum_{n=0}^{q}\left( -1\right)
^{n}L_{tsn}^{k}\left( x\right) $. We give conditions on the parameters $%
s,t,k $ for express these sums as some proposed linear combinations of the $%
s $-Fibopolynomials $\binom{q+m}{tk}_{F_{s}\left( x\right) }$, $m=1,2,\ldots
,tk $.
\end{abstract}

\section{Introduction}

We use $\mathbb{N}$ for the natural numbers and $\mathbb{N}^{\prime }$ for $%
\mathbb{N\cup }\left\{ 0\right\} $.

We follow the standard notation $F_{n}\left( x\right) $ for Fibonacci
polynomials and $L_{n}\left( x\right) $ for Lucas polynomials. Binet's
formulas 
\begin{equation}
F_{n}\left( x\right) =\frac{1}{\sqrt{x^{2}+4}}\left( \alpha ^{n}\left(
x\right) -\beta ^{n}\left( x\right) \right) \ \ \ \text{and}\ \ \
L_{n}\left( x\right) =\alpha ^{n}\left( x\right) +\beta ^{n}\left( x\right) ,
\label{1.1}
\end{equation}%
where 
\begin{equation}
\alpha \left( x\right) =\frac{1}{2}\left( x+\sqrt{x^{2}+4}\right) \text{ \ \
\ and \ \ \ }\beta \left( x\right) =\frac{1}{2}\left( x-\sqrt{x^{2}+4}%
\right) \text{,}  \label{1.2}
\end{equation}%
will be used extensively (without further comments). We will use also the
identities%
\begin{eqnarray}
\frac{F_{\left( 2p-1\right) s}\left( x\right) }{F_{s}\left( x\right) }
&=&\sum_{k=0}^{p-1}\left( -1\right) ^{sk}L_{2\left( p-k-1\right) s}\left(
x\right) -\left( -1\right) ^{s\left( p-1\right) },  \label{1.3} \\
\frac{F_{2ps}\left( x\right) }{F_{s}\left( x\right) } &=&\sum_{k=0}^{p-1}%
\left( -1\right) ^{sk}L_{\left( 2p-2k-1\right) s}\left( x\right) ,
\label{1.4} \\
F_{M}\left( x\right) F_{N}\left( x\right) -F_{M+K}\left( x\right)
F_{N-K}\left( x\right) &=&\left( -1\right) ^{N-K}F_{M+K-N}\left( x\right)
F_{K}\left( x\right) ,  \label{1.10}
\end{eqnarray}%
where $p\in \mathbb{N}$ in (\ref{1.3}) and (\ref{1.4}), and $M,N,K\in 
\mathbb{Z}$ in (\ref{1.10}). (Identity (\ref{1.10}) is a version of the
so-called \textquotedblleft index-reduction formula\textquotedblright ; see 
\cite{Jh} for the case $x=1$.) Two variants of (\ref{1.10}) we will use in
section \ref{Sec5} are%
\begin{eqnarray}
F_{M}\left( x\right) L_{N}\left( x\right) -F_{M+K}\left( x\right)
L_{N-K}\left( x\right) &=&\left( -1\right) ^{N-K+1}L_{M+K-N}\left( x\right)
F_{K}\left( x\right) ,  \label{1.101} \\
\left( x^{2}+4\right) F_{M}\left( x\right) F_{N}\left( x\right)
-L_{M+K}\left( x\right) L_{N-K}\left( x\right) &=&\left( -1\right)
^{N-K+1}L_{M+K-N}\left( x\right) L_{K}\left( x\right) .  \label{1.102}
\end{eqnarray}

Given $n\in \mathbb{N}^{\prime }$ and $k\in \left\{ 0,1,\ldots ,n\right\} $,
the $s$-\textit{Fibopolynomial} $\binom{n}{k}_{F_{s}\left( x\right) }$ is
defined by $\binom{n}{0}_{F_{s}\left( x\right) }=\binom{n}{n}_{F_{s}\left(
x\right) }=1$, and%
\begin{equation}
\binom{n}{k}_{F_{s}\left( x\right) }=\frac{F_{sn}\left( x\right) F_{s\left(
n-1\right) }\left( x\right) \cdots F_{s\left( n-k+1\right) }\left( x\right) 
}{F_{s}\left( x\right) F_{2s}\left( x\right) \cdots F_{ks}\left( x\right) }.
\label{1.12}
\end{equation}

(These objects were already used in \cite{Pi3}, where we called them
\textquotedblleft $s$-polyfibonomials\textquotedblright . However, we think
now that \textquotedblleft $s$-Fibopolynomials\textquotedblright\ is a
better name to describe them.)

Plainly we have symmetry for $s$-Fibopolynomials: $\binom{n}{k}_{F_{s}\left(
x\right) }=\binom{n}{n-k}_{F_{s}\left( x\right) }$. We can use the identity%
\begin{equation*}
F_{s\left( n-k\right) +1}\left( x\right) F_{sk}\left( x\right)
+F_{sk-1}\left( x\right) F_{s\left( n-k\right) }\left( x\right)
=F_{sn}\left( x\right) ,
\end{equation*}%
(which comes from (\ref{1.10}) with $M=sn$, $N=1$ and $K=-sk+1$), to
conclude that%
\begin{equation}
\binom{n}{k}_{F\!_{s}\!\left( x\right) }=F_{s\left( n-k\right) +1}\left(
x\right) \binom{n-1}{k-1}_{F\!_{s}\!\left( x\right) }+F_{sk-1}\left(
x\right) \binom{n-1}{k}_{F\!_{s}\!\left( x\right) }.  \label{1.14}
\end{equation}

Formula (\ref{1.14}), together with a simple induction argument, shows that $%
s$-Fibopolynomials are indeed polynomials (with $\deg \binom{n}{k}%
_{F\!_{s}\!\left( x\right) }=sk\left( n-k\right) $). The case $s=x=1$
corresponds to Fibonomials $\binom{n}{k}_{F\!\!}$ \negthinspace , introduced
by V. E. Hoggatt, Jr. \cite{Hog} in 1967 (see also \cite{T-F}), and the case 
$x=1$ corresponds to $s$-Fibonomials $\binom{n}{k}_{F_{s}}$, first mentioned
also in \cite{Hog}, and studied recently in \cite{Pi2}. We comment in
passing that Fibonomials are important mathematical objects involved in many
interesting research works during the last few decades (see \cite{Go}, \cite%
{Ki0}, \cite{Ki05}, \cite{Ki2}, \cite{Tr}, to mention some). 

The well-known identity 
\begin{equation}
\sum_{n=0}^{q}F_{n}^{2}=F_{q}F_{q+1}=\binom{q+1}{2}_{F},  \label{1.15}
\end{equation}%
was the initial motivation for this work. We will see that (\ref{1.15}) is
just a particular case of the following polynomial identities ((\ref{4.27})
in section \ref{Sec4}) 
\begin{equation*}
\left( -1\right) ^{\left( s+1\right) q}L_{s}\left( x\right)
\sum_{n=0}^{q}\left( -1\right) ^{\left( s+1\right) n}F_{sn}^{2}\left(
x\right) =\left( -1\right) ^{sq}F_{s}\left( x\right) \sum_{n=0}^{q}\left(
-1\right) ^{sn}F_{2sn}\left( x\right) =F_{s\left( q+1\right) }\left(
x\right) F_{sq}\left( x\right) .
\end{equation*}

To find closed formulas for sums of powers of Fibonacci and Lucas numbers $%
\sum_{n=0}^{q}F_{n}^{k}$ and $\sum_{n=0}^{q}L_{n}^{k}$, and for the
corresponding alternating sums of powers $\sum_{n=0}^{q}\left( -1\right)
^{n}F_{n}^{k}$ and $\sum_{n=0}^{q}\left( -1\right) ^{n}L_{n}^{k}$, is a
challenging problem that has been in the interest of many mathematicians
along the years (see \cite{Be}, \cite{C-He}, \cite{Mel1}, \cite{Mel2}, \cite%
{O-N}, \cite{Oz}, \cite{Pr}, to mention some). There are also some works
considering variants of these sums and/or generalizations (in some sense) of
them (see \cite{C-H}, \cite{Ki1}, \cite{Mel3}, \cite{S}, among many others).
This work presents, on one hand, a generalization of the problem mentioned
above, considering Fibonacci and Lucas polynomials (instead of numbers) and
involving more parameters in the sums. On the other hand, we are not
interested in any closed formulas for these sums, but only in sums that can
be written as certain linear combinations of certain $s$-Fibopolynomials (as
in (\ref{1.15})). More precisely, in this work we obtain sufficient
conditions (on the positive integer parameters $t,k,s$), for the polynomial
sums of powers $\sum_{n=0}^{q}F_{tsn}^{k}\left( x\right) $, $%
\sum_{n=0}^{q}L_{tsn}^{k}\left( x\right) $, $\sum_{n=0}^{q}\left( -1\right)
^{n}F_{tsn}^{k}\left( x\right) $ and $\sum_{n=0}^{q}\left( -1\right)
^{n}L_{tsn}^{k}\left( x\right) $, can be expressed as linear combinations of
the $s$-Fibopolynomials $\binom{q+m}{tk}_{F_{s}\left( x\right) }$, $%
m=1,2,\ldots ,tk$, according to some proposed expressions ((\ref{3.3}), (\ref%
{3.23}), (\ref{4.5}) and (\ref{4.6}), respectively). (We conjecture that
these sufficient conditions are also necessary, see remark \ref{Remark}.) In
section \ref{Sec2} we recall some facts about $Z$ transform, since some
results related to the $Z$ transform of the sequences $\left\{
F_{tsn}^{k}\left( x\right) \right\} _{n=0}^{\infty }$ and $\left\{
L_{tsn}^{k}\left( x\right) \right\} _{n=0}^{\infty }$ (obtained in a
previous work) are the starting point of the results in this work. The main
results are presented in sections \ref{Sec3} and \ref{Sec4}. Propositions %
\ref{Prop3.1} and \ref{Prop3.2} in section \ref{Sec3} contain, respectively,
sufficient conditions on the positive integers $t,k,s$ for the sums of
powers $\sum_{n=0}^{q}F_{tsn}^{k}\left( x\right) $ and $%
\sum_{n=0}^{q}L_{tsn}^{k}\left( x\right) $ can be written as linear
combinations of the mentioned $s$-Fibopolynomials, and propositions \ref%
{Prop4.1} and \ref{Prop4.2} in section \ref{Sec4} contain, respectively,
sufficient conditions on the positive integers $t,k,s$ for the alternating
sums of powers $\sum_{n=0}^{q}\left( -1\right) ^{n}F_{tsn}^{k}\left(
x\right) $ and $\sum_{n=0}^{q}\left( -1\right) ^{n}L_{tsn}^{k}\left(
x\right) $ can be written as linear combinations of those $s$%
-Fibopolynomials. Surprisingly, there are some intersections on the
conditions on $t$ and $k$ in propositions \ref{Prop3.1} and \ref{Prop4.1}
(and also in propositions \ref{Prop3.2} and \ref{Prop4.2}), allowing us to
write results for sums of powers of the form $\sum_{n=0}^{q}\left( -1\right)
^{sn}F_{tsn}^{k}\left( x\right) $ or $\sum_{n=0}^{q}\left( -1\right)
^{\left( s+1\right) n}F_{tsn}^{k}\left( x\right) $ (and similar sums for
Lucas polynomials), that work at the same time for sums $%
\sum_{n=0}^{q}F_{tsn}^{k}\left( x\right) $ and alternating sums $%
\sum_{n=0}^{q}\left( -1\right) ^{n}F_{tsn}^{k}\left( x\right) $ as well,
depending on the parity of $s$. These results are presented in section \ref%
{Sec4}: corollaries \ref{Cor4.1} (for the Fibonacci case) and \ref{Cor4.2}
(for the Lucas case). Finally, in section \ref{Sec5} we show some examples
of identities obtained as derivatives of some of the results obtained in
previous sections.

\section{\label{Sec2}Preliminaries}

The $Z$ transform maps complex sequences $\left\{ a_{n}\right\}
_{n=0}^{\infty }$ into holomorphic functions $A:U\subset \mathbb{%
C\rightarrow C}$, defined by the Laurent series $A\left( z\right)
=\sum_{n=0}^{\infty }a_{n}z^{-n}$ (also written as $\mathcal{Z}\left(
a_{n}\right) $, defined outside the closure of the disk of convergence of
the Taylor series $\sum_{n=0}^{\infty }a_{n}z^{n}$). We also write $a_{n}=%
\mathcal{Z}^{-1}\left( A\left( z\right) \right) $ and we say the the
sequence $\left\{ a_{n}\right\} _{n=0}^{\infty }$ is the inverse $Z$
transform of $A\left( z\right) $. Some basic facts we will need are the
following:

(a) $\mathcal{Z}$ is linear and injective (same for $\mathcal{Z}^{-1}$).

(b) If $\left\{ a_{n}\right\} _{n=0}^{\infty }$ is a sequence with $Z$
transform $A\left( z\right) $, then the $Z$ transform of the sequence $%
\left\{ \left( -1\right) ^{n}a_{n}\right\} _{n=0}^{\infty }$ is 
\begin{equation}
\mathcal{Z}\left( \left( -1\right) ^{n}a_{n}\right) =A\left( -z\right) .
\label{2.3}
\end{equation}

(c) If $\left\{ a_{n}\right\} _{n=0}^{\infty }$ is a sequence with $Z$
transform $A\left( z\right) $, then the $Z$ transform of the sequence $%
\left\{ na_{n}\right\} _{n=0}^{\infty }$ is 
\begin{equation}
\mathcal{Z}\left( na_{n}\right) =-z\frac{d}{dz}A\left( z\right) .
\label{2.31}
\end{equation}

Plainly we have (for given $\lambda \in \mathbb{C}$, $\lambda \neq 0$)%
\begin{equation}
\mathcal{Z}\left( \lambda ^{n}\right) =\frac{z}{z-\lambda }.  \label{2.1}
\end{equation}

For example, if $t,k\in \mathbb{N}^{\prime }$ are given, we can write the
generic term of the sequence $\left\{ F_{tsn}^{k}\left( x\right) \right\}
_{n=0}^{\infty }$ as%
\begin{eqnarray}
F_{tsn}^{k}\left( x\right) &=&\left( \frac{1}{\sqrt{x^{2}+4}}\left( \alpha
^{tsn}\left( x\right) -\beta ^{tsn}\left( x\right) \right) \right) ^{k}
\label{2.4} \\
&=&\frac{1}{\left( x^{2}+4\right) ^{\frac{k}{2}}}\sum_{l=0}^{k}\binom{k}{l}%
\left( -1\right) ^{k-l}\left( \alpha ^{tsl}\left( x\right) \beta ^{ts\left(
k-l\right) }\left( x\right) \right) ^{n}.  \notag
\end{eqnarray}

The linearity of $\mathcal{Z}$ and (\ref{2.1}) give us%
\begin{equation}
\mathcal{Z}\left( F_{tsn}^{k}\left( x\right) \right) =\frac{1}{\left(
x^{2}+4\right) ^{\frac{k}{2}}}\sum_{l=0}^{k}\binom{k}{l}\left( -1\right)
^{k-l}\frac{z}{z-\alpha ^{tsl}\left( x\right) \beta ^{ts\left( k-l\right)
}\left( x\right) }.  \label{2.5}
\end{equation}

Similarly, since the generic term of the sequence $\left\{ L_{tsn}^{k}\left(
x\right) \right\} _{n=0}^{\infty }$ can be expressed as%
\begin{equation}
L_{tsn}^{k}\left( x\right) =\left( \alpha ^{tsn}\left( x\right) +\beta
^{tsn}\left( x\right) \right) ^{k}=\sum_{l=0}^{k}\binom{k}{l}\left( \alpha
^{tsl}\left( x\right) \beta ^{ts\left( k-l\right) }\left( x\right) \right)
^{n},  \label{2.6}
\end{equation}%
we have that%
\begin{equation}
\mathcal{Z}\left( L_{tsn}^{k}\left( x\right) \right) =\sum_{l=0}^{k}\binom{k%
}{l}\frac{z}{z-\alpha ^{tsl}\left( x\right) \beta ^{ts\left( k-l\right)
}\left( x\right) }.  \label{2.7}
\end{equation}

Observe that formulas 
\begin{equation}
\mathcal{Z}\left( F_{n}\left( x\right) \right) =\frac{z}{z^{2}-xz-1}\text{ \
\ \ \ and \ \ \ \ \ }\mathcal{Z}\left( L_{n}\left( x\right) \right) =\frac{%
z\left( 2z-x\right) }{z^{2}-xz-1},  \label{2.75}
\end{equation}%
are the simplest cases ($k=t=s=1$) of (\ref{2.5}) and (\ref{2.7}),
respectively.

In a recent work \cite{Pi4} (inspired by \cite{Hor}, among others), we
proved that expressions (\ref{2.5}) and (\ref{2.7}) can be written in a
special form. The result is that (\ref{2.5}) can be written as%
\begin{equation}
\mathcal{Z}\left( F_{tsn}^{k}\left( x\right) \right) =z\frac{%
\sum_{i=0}^{tk}\sum_{j=0}^{i}\left( -1\right) ^{\frac{\left(
sj+2(s+1)\right) \left( j+1\right) }{2}}\dbinom{tk+1}{j}_{F_{s}\left(
x\right) }F_{ts\left( i-j\right) }^{k}\left( x\right) z^{tk-i}}{%
\sum_{i=0}^{tk+1}\left( -1\right) ^{\frac{\left( si+2(s+1)\right) \left(
i+1\right) }{2}}\dbinom{tk+1}{i}_{F_{s}\left( x\right) }z^{tk+1-i}},
\label{2.71}
\end{equation}%
and (\ref{2.7}) can be written as%
\begin{equation}
\mathcal{Z}\left( L_{tsn}^{k}\left( x\right) \right) =z\frac{%
\sum_{i=0}^{tk}\sum_{j=0}^{i}\left( -1\right) ^{\frac{\left(
sj+2(s+1)\right) \left( j+1\right) }{2}}\dbinom{tk+1}{j}_{F_{s}\left(
x\right) }L_{ts\left( i-j\right) }^{k}\left( x\right) z^{tk-i}}{%
\sum_{i=0}^{tk+1}\left( -1\right) ^{\frac{\left( si+2(s+1)\right) \left(
i+1\right) }{2}}\dbinom{tk+1}{i}_{F_{s}\left( x\right) }z^{tk+1-i}}.
\label{2.72}
\end{equation}

From (\ref{2.71}) and (\ref{2.72}) we obtained that $F_{tsn}^{k}\left(
x\right) $ and $L_{tsn}^{k}\left( x\right) $ can be expressed as linear
combinations of the $s$-Fibopolynomials $\binom{n+tk-i}{tk}_{F_{s}\left(
x\right) }$, $i=0,1,\ldots ,tk$, according to%
\begin{equation}
F_{tsn}^{k}\left( x\right) =\left( -1\right)
^{s+1}\sum_{i=0}^{tk}\sum_{j=0}^{i}\left( -1\right) ^{\frac{\left(
sj+2(s+1)\right) \left( j+1\right) }{2}}\dbinom{tk+1}{j}_{F_{s}\left(
x\right) }F_{ts\left( i-j\right) }^{k}\left( x\right) \dbinom{n+tk-i}{tk}%
_{F_{s}\left( x\right) },  \label{2.81}
\end{equation}%
and%
\begin{equation}
L_{tsn}^{k}\left( x\right) =\left( -1\right)
^{s+1}\sum_{i=0}^{tk}\sum_{j=0}^{i}\left( -1\right) ^{\frac{\left(
sj+2(s+1)\right) \left( j+1\right) }{2}}\dbinom{tk+1}{j}_{F_{s}\left(
x\right) }L_{ts\left( i-j\right) }^{k}\left( x\right) \dbinom{n+tk-i}{tk}%
_{F_{s}\left( x\right) }.  \label{2.82}
\end{equation}

The denominator in (\ref{2.71}) (or (\ref{2.72})) is a $\left( tk+1\right) $%
-th degree $z$-polynomial, which we denote as $D_{s,tk+1}\left( x,z\right) $%
, that can be factored as%
\begin{equation}
\sum_{i=0}^{tk+1}\left( -1\right) ^{\frac{\left( si+2(s+1)\right) \left(
i+1\right) }{2}}\dbinom{tk+1}{i}_{F_{s}\left( x\right) }z^{tk+1-i}=\left(
-1\right) ^{s+1}\dprod\limits_{j=0}^{tk}\left( z-\alpha ^{sj}\left( x\right)
\beta ^{s\left( tk-j\right) }\left( x\right) \right) .  \label{2.10}
\end{equation}

(See proposition 1 in \cite{Pi4}.) Moreover, if $tk$ is even, $tk=2p$ say,
then (\ref{2.10}) can be written as%
\begin{equation}
D_{s,2p+1}\left( x;z\right) =\left( -1\right) ^{s+1}\left( z-\left(
-1\right) ^{sp}\right) \prod_{j=0}^{p-1}\left( z^{2}-\left( -1\right)
^{sj}L_{2s\left( p-j\right) }\left( x\right) z+1\right) ,  \label{2.11}
\end{equation}%
and if $tk$ is odd, $tk=2p-1$ say, we have%
\begin{equation}
D_{s,2p}\left( x;z\right) =\left( -1\right) ^{s+1}\prod_{j=0}^{p-1}\left(
z^{2}-\left( -1\right) ^{sj}L_{s\left( 2p-1-2j\right) }\left( x\right)
z+\left( -1\right) ^{\left( 2p-1\right) s}\right) .  \label{2.12}
\end{equation}

(See (40) and (41) in \cite{Pi4}.)

\section{\label{Sec3}The main results (I)}

Let us consider first the Fibonacci case. From (\ref{2.81}) we can write the
sum $\sum_{n=0}^{q}F_{tsn}^{k}\left( x\right) $ in terms of a sum of $s$%
-Fibopolynomials in a trivial way, namely%
\begin{equation}
\sum_{n=0}^{q}F_{tsn}^{k}\left( x\right) =\left( -1\right)
^{s+1}\sum_{i=0}^{tk}\sum_{j=0}^{i}\left( -1\right) ^{\frac{\left(
sj+2(s+1)\right) \left( j+1\right) }{2}}\dbinom{tk+1}{j}_{F_{s}\left(
x\right) }F_{ts\left( i-j\right) }^{k}\left( x\right) \sum_{n=0}^{q}\dbinom{%
n+tk-i}{tk}_{F_{s}\left( x\right) }.  \label{3.1}
\end{equation}

The point is that we can write (\ref{3.1}) as%
\begin{eqnarray}
\sum_{n=0}^{q}F_{tsn}^{k}\left( x\right) &=&\left( -1\right)
^{s+1}\sum_{m=1}^{tk}\sum_{i=0}^{tk-m}\sum_{j=0}^{i}\left( -1\right) ^{\frac{%
\left( sj+2(s+1)\right) \left( j+1\right) }{2}}\dbinom{tk+1}{j}_{F_{s}\left(
x\right) }F_{ts\left( i-j\right) }^{k}\left( x\right) \dbinom{q+m}{tk}%
_{F_{s}\left( x\right) }  \label{3.2} \\
&&+\left( -1\right) ^{s+1}\sum_{i=0}^{tk}\sum_{j=0}^{i}\left( -1\right) ^{%
\frac{\left( sj+2(s+1)\right) \left( j+1\right) }{2}}\dbinom{tk+1}{j}%
_{F_{s}\left( x\right) }F_{ts\left( i-j\right) }^{k}\left( x\right)
\sum_{n=0}^{q}\dbinom{n}{tk}_{F_{s}\left( x\right) }.  \notag
\end{eqnarray}

Expression (\ref{3.2}) tells us that the sum $\sum_{n=0}^{q}F_{tsn}^{k}%
\left( x\right) $ can be written as a linear combination of the $s$%
-Fibopolynomials $\binom{q+m}{tk}_{F_{s}\left( x\right) }$, $m=1,2,\ldots
,tk $, according to 
\begin{equation}
\sum_{n=0}^{q}F_{tsn}^{k}\left( x\right) =\left( -1\right)
^{s+1}\sum_{m=1}^{tk}\sum_{i=0}^{tk-m}\sum_{j=0}^{i}\left( -1\right) ^{\frac{%
\left( sj+2(s+1)\right) \left( j+1\right) }{2}}\dbinom{tk+1}{j}_{F_{s}\left(
x\right) }F_{ts\left( i-j\right) }^{k}\left( x\right) \dbinom{q+m}{tk}%
_{F_{s}\left( x\right) },  \label{3.3}
\end{equation}%
if and only if 
\begin{equation}
\sum_{i=0}^{tk}\sum_{j=0}^{i}\left( -1\right) ^{\frac{\left(
sj+2(s+1)\right) \left( j+1\right) }{2}}\dbinom{tk+1}{j}_{F_{s}\left(
x\right) }F_{ts\left( i-j\right) }^{k}\left( x\right) =0.  \label{3.4}
\end{equation}

Observe that from (\ref{2.5}) and (\ref{2.71}) we can write%
\begin{eqnarray}
&&\sum_{i=0}^{tk}\sum_{j=0}^{i}\left( -1\right) ^{\frac{\left(
sj+2(s+1)\right) \left( j+1\right) }{2}}\dbinom{tk+1}{j}_{F_{s}\left(
x\right) }F_{ts\left( i-j\right) }^{k}\left( x\right) z^{tk-i}  \label{3.5}
\\
&=&\!\frac{1}{\left( x^{2}+4\right) ^{\frac{k}{2}}}\!\left( \sum_{l=0}^{k}%
\binom{k}{l}\!\left( -1\right) ^{k-l}\frac{1}{z-\alpha ^{lts}\!\left(
x\right) \beta ^{\left( k-l\right) ts}\!\left( x\right) }\right) \!\!\left(
\sum_{i=0}^{tk+1}\left( -1\right) ^{\frac{\left( si+2(s+1)\right) \left(
i+1\right) }{2}}\dbinom{tk+1}{i}_{F_{s}\left( x\right) }\!z^{tk+1-i}\right) .
\notag
\end{eqnarray}

Let us consider the factors in parentheses of the right-hand side of (\ref%
{3.5}), namely%
\begin{equation}
\Pi _{1}\left( x,z\right) =\sum_{l=0}^{k}\binom{k}{l}\left( -1\right) ^{k-l}%
\frac{1}{z-\alpha ^{lts}\left( x\right) \beta ^{\left( k-l\right) ts}\left(
x\right) },  \label{3.6}
\end{equation}%
and 
\begin{equation}
\Pi _{2}\left( x,z\right) =\sum_{i=0}^{tk+1}\left( -1\right) ^{\frac{\left(
si+2(s+1)\right) \left( i+1\right) }{2}}\dbinom{tk+1}{i}_{F_{s}\left(
x\right) }z^{tk+1-i}.  \label{3.7}
\end{equation}

Clearly any of the conditions%
\begin{equation}
\Pi _{1}\left( x,1\right) =0,  \label{3.81}
\end{equation}%
or 
\begin{equation}
\begin{array}{ccc}
\Pi _{1}\left( x,1\right) <\infty & \text{and} & \Pi _{2}\left( x,1\right)
=0.%
\end{array}
\label{3.82}
\end{equation}%
imply (\ref{3.4}).

\begin{proposition}
\label{Prop3.1}The sum $\sum_{n=1}^{q}F_{tsn}^{k}\left( x\right) $ can be
written as a linear combination of the $s$-Fibopolynomials $\binom{q+m}{tk}%
_{F_{s}\left( x\right) }$, $m=1,2,\ldots ,tk$, according to (\ref{3.3}), in
the following cases%
\begin{equation*}
\begin{tabular}{|c||c|c|c|}
\hline
& $t$ & $k$ & $s$ \\ \hline\hline
(a) & even & odd & even \\ \hline
(b) & odd & $\equiv 2\func{mod}4$ & odd \\ \hline
(c) & $\equiv 0\func{mod}4$ & odd & any \\ \hline
\end{tabular}%
\end{equation*}
\end{proposition}

\begin{proof}
Observe that in each of the three cases the product $tk$ is even. Then,
according to (\ref{2.11}) we can write%
\begin{equation}
\Pi _{2}\left( x,z\right) =\left( -1\right) ^{s+1}\left( z-\left( -1\right)
^{\frac{kts}{2}}\right) \prod\limits_{j=0}^{\frac{tk}{2}-1}\left(
z^{2}-\left( -1\right) ^{sj}L_{2s\left( \frac{tk}{2}-j\right) }\left(
x\right) z+1\right) .  \label{3.9}
\end{equation}

(a) Let us suppose that $t$ is even, $k$ is odd and $s$ is even. In this
case the factor $\left( z-\left( -1\right) ^{\frac{kts}{2}}\right) $ of the
right-hand side of (\ref{3.9}) is $\left( z-1\right) $, so we have $\Pi
_{2}\left( x,1\right) =0$. It remains to check that $\Pi _{1}\left(
x,1\right) $ is finite. In fact, by writing $k$ as $2k-1$, and using that $t$
and $s$ are even, one can check that%
\begin{equation}
\Pi _{1}\left( x,1\right) =\sqrt{x^{2}+4}\sum_{l=0}^{k-1}\binom{2k-1}{l}%
\left( -1\right) ^{l+1}\frac{F_{\left( 2k-1-2l\right) ts}\left( x\right) }{%
2-L_{\left( 2k-1-2l\right) ts}\left( x\right) },  \label{3.10}
\end{equation}%
so we have that $\Pi _{1}\left( x,1\right) $ is finite, and then the
right-hand side of (\ref{3.5}) is equal to zero when $z=1$, as wanted.

(b) Suppose now that $t$ is odd, $k\equiv 2\func{mod}4$ and that $s$ is odd.
In this case the factor $\left( z-\left( -1\right) ^{\frac{kts}{2}}\right) $
of the right-hand side of (\ref{3.9}) is $\left( z+1\right) $, so $\Pi
_{2}\left( x,1\right) \neq 0$. However, by writing $k$ as $2\left(
2k-1\right) $ and using that $t$ and $s$ are odd, we can see that 
\begin{equation*}
\Pi _{1}\left( x,1\right) =\sum_{l=0}^{2k-2}\binom{2\left( 2k-1\right) }{l}%
\left( -1\right) ^{l}-\frac{1}{2}\binom{2\left( 2k-1\right) }{2k-1}=0.
\end{equation*}

Thus, the right-hand side of (\ref{3.5}) is equal to $0$ when $z=1$, as
wanted.

(c) Let us suppose that $t\equiv 0\func{mod}4$, $k$ is odd, and $s$ is any
positive integer. In this case the factor $\left( z-\left( -1\right) ^{\frac{%
kts}{2}}\right) $ of the right-hand side of (\ref{3.9}) is $\left(
z-1\right) $, so we have $\Pi _{2}\left( x,1\right) =0$. By writing $k$ as $%
2k-1$, and using that $t$ is multiple of $4$, we can see that formula (\ref%
{3.10}) is valid for any $s\in \mathbb{N}$, so we conclude that $\Pi
_{1}\left( x,1\right) $ is finite. Thus the right-hand side of (\ref{3.5})
is $0$ when $z=1$, as wanted.
\end{proof}

An example from the case (c) of proposition \ref{Prop3.1} is the following
identity (corresponding to $t=4$ and $k=1$), valid for any $s\in \mathbb{N}$%
\begin{equation}
\sum_{n=0}^{q}F_{4sn}\left( x\right) =F_{4s}\left( x\right) \left( \dbinom{%
q+1}{4}_{F_{s}\left( x\right) }+\left( -1\right) ^{s+1}L_{2s}\left( x\right) 
\dbinom{q+2}{4}_{F_{s}\left( x\right) }+\dbinom{q+3}{4}_{F_{s}\left(
x\right) }\right) .  \label{3.18}
\end{equation}

\begin{remark}
\label{Remark}A natural question about proposition \ref{Prop3.1} is wether
the given conditions on $t$, $k$ and $s$ are also necessary (for expressing
the sum $\sum_{n=1}^{q}F_{tsn}^{k}\left( x\right) $ as a linear combination
of the $s$-Fibopolynomials $\binom{q+m}{tk}_{F_{s}\left( x\right) }$, $%
m=1,2,\ldots ,tk$, according to (\ref{3.3})). We believe that the answer is
yes, and we think that this conjecture (together with similar conjectures in
propositions \ref{Prop3.2}, \ref{Prop4.1} and \ref{Prop4.2}) can be a good
topic for a future work. Nevertheless, we would like to make some comments
about this point in the case of proposition \ref{Prop3.1}. The cases where
we do not have the conditions on $t,k,s$ stated in proposition \ref{Prop3.1}
are the following%
\begin{equation*}
\begin{tabular}{|c||c|c|c|c|c|c|c|}
\hline
& (i) & (ii) & (iii) & (iv) & (v) & (vi) & (vii) \\ \hline
$t$ & e & e & o & $\equiv 2\func{mod}4$ & o & o & o \\ \hline
$k$ & e & e & e & o & $\equiv 0\func{mod}4$ & o & o \\ \hline
$s$ & e & o & e & o & o & e & o \\ \hline
\end{tabular}%
\end{equation*}

Thus, in order to prove the necessity of the mentioned conditions we need to
show that (\ref{3.4}) \textit{does not hold} in each of these 7 cases. For
example, the case $t=1$ and $k=3$ is included in (vi) and (vii). In this
case the left-hand side of (\ref{3.4}) is (for any $s\in \mathbb{N}$)%
\begin{equation*}
\sum_{i=0}^{3}\sum_{j=0}^{i}\left( -1\right) ^{\frac{\left( sj+2(s+1)\right)
\left( j+1\right) }{2}}\dbinom{4}{j}_{F_{s}\left( x\right) }F_{s\left(
i-j\right) }^{3}\left( x\right) =-F_{s}^{3}\left( x\right) \left(
2L_{s}\left( x\right) +1+\left( -1\right) ^{s}\right) .
\end{equation*}

That is, (\ref{3.4}) does not hold, which means that the sum of cubes $%
\sum_{n=0}^{q}F_{sn}^{3}\left( x\right) $ can not be written as the linear
combination of the $s$-Fibopolynomials $\binom{q+1}{3}_{F_{s}\left( x\right)
}$ $,$ $\binom{q+2}{3}_{F_{s}\left( x\right) }$ and $\binom{q+3}{3}%
_{F_{s}\left( x\right) }$ proposed in (\ref{3.3}). However, it is known that 
$\sum_{n=0}^{q}F_{n}^{3}=\frac{1}{10}\left( F_{3q+2}-6\left( -1\right)
^{q}F_{q-1}+5\right) $ (see \cite{Be}). In fact, the case of sums of odd
powers of Fibonacci and Lucas numbers has been considered for several
authors (see \cite{C-He}, \cite{Oz}, \cite{Pr}). It turns out that some of
their nice results belong to some of the cases (i) to (vii) above, so they
can not be written as in (\ref{3.3}).
\end{remark}

Let us consider now the case of sums of powers of Lucas polynomials. From (%
\ref{2.82}) we see that%
\begin{equation}
\left( -1\right) ^{s+1}\sum_{n=0}^{q}L_{tsn}^{k}\left( x\right)
=\sum_{i=0}^{tk}\sum_{j=0}^{i}\left( -1\right) ^{\frac{\left(
sj+2(s+1)\right) \left( j+1\right) }{2}}\dbinom{tk+1}{j}_{F_{s}\left(
x\right) }L_{ts\left( i-j\right) }^{k}\left( x\right) \sum_{n=0}^{q}\dbinom{%
n+tk-i}{tk}_{F_{s}\left( x\right) },  \label{3.21}
\end{equation}%
which can be written as%
\begin{eqnarray}
\sum_{n=0}^{q}L_{tsn}^{k}\left( x\right) &=&\left( -1\right)
^{s+1}\sum_{m=1}^{tk}\sum_{i=0}^{tk-m}\sum_{j=0}^{i}\left( -1\right) ^{\frac{%
\left( sj+2(s+1)\right) \left( j+1\right) }{2}}\dbinom{tk+1}{j}_{F_{s}\left(
x\right) }L_{ts\left( i-j\right) }^{k}\left( x\right) \dbinom{q+m}{tk}%
_{F_{s}\left( x\right) }  \label{3.22} \\
&&+\left( -1\right) ^{s+1}\sum_{i=0}^{tk}\sum_{j=0}^{i}\left( -1\right) ^{%
\frac{\left( sj+2(s+1)\right) \left( j+1\right) }{2}}\dbinom{tk+1}{j}%
_{F_{s}\left( x\right) }L_{ts\left( i-j\right) }^{k}\left( x\right)
\sum_{n=0}^{q}\dbinom{n}{tk}_{F_{s}\left( x\right) }.  \notag
\end{eqnarray}

Expression (\ref{3.22}) tells us that the sum $\sum_{n=0}^{q}L_{tsn}^{k}%
\left( x\right) $ can be written as a linear combination of the $s$%
-Fibopolynomials $\binom{q+m}{tk}_{F_{s}\left( x\right) }$, $m=1,2,\ldots
,tk $, according to 
\begin{equation}
\sum_{n=0}^{q}L_{tsn}^{k}\left( x\right) =\left( -1\right)
^{s+1}\sum_{m=1}^{tk}\sum_{i=0}^{tk-m}\sum_{j=0}^{i}\left( -1\right) ^{\frac{%
\left( sj+2(s+1)\right) \left( j+1\right) }{2}}\dbinom{tk+1}{j}_{F_{s}\left(
x\right) }L_{ts\left( i-j\right) }^{k}\left( x\right) \dbinom{q+m}{tk}%
_{F_{s}\left( x\right) },  \label{3.23}
\end{equation}%
if and only if 
\begin{equation}
\sum_{i=0}^{tk}\sum_{j=0}^{i}\left( -1\right) ^{\frac{\left(
sj+2(s+1)\right) \left( j+1\right) }{2}}\dbinom{tk+1}{j}_{F_{s}\left(
x\right) }L_{ts\left( i-j\right) }^{k}\left( x\right) =0.  \label{3.24}
\end{equation}

From (\ref{2.7}) and (\ref{2.72}) we can write%
\begin{eqnarray}
&&\sum\limits_{i=0}^{tk}\!\sum\limits_{j=0}^{i}\left( -1\right) ^{\frac{%
\left( sj+2(s+1)\right) \left( j+1\right) }{2}}\!\dbinom{tk+1}{j}%
_{F_{s}\left( x\right) }\!L_{ts\left( i-j\right) }^{k}\!\left( x\right)
z^{tk-i}  \label{3.25} \\
&=&\left( \sum\limits_{i=0}^{tk+1}\left( -1\right) ^{\frac{\left(
si+2(s+1)\right) \left( i+1\right) }{2}}\dbinom{tk+1}{i}_{F_{s}\left(
x\right) }z^{tk+1-i}\right) \sum_{l=0}^{k}\binom{k}{l}\frac{1}{z-\alpha
^{lts}\left( x\right) \beta ^{\left( k-l\right) ts}\left( x\right) }.  \notag
\end{eqnarray}

We have again the factor $\Pi _{2}\left( x,z\right) $ considered in the
Fibonacci case (see (\ref{3.7})), and the factor%
\begin{equation}
\widetilde{\Pi }_{1}\left( x,z\right) =\sum_{l=0}^{k}\binom{k}{l}\frac{1}{%
z-\alpha ^{lts}\left( x\right) \beta ^{\left( k-l\right) ts}\left( x\right) }%
.  \label{3.26}
\end{equation}

Plainly any of the conditions: (a) $\widetilde{\Pi }_{1}\left( x,1\right) =0$%
, or, (b) $\widetilde{\Pi }_{1}\left( x,1\right) <\infty $ and $\Pi
_{2}\left( x,1\right) =0$, imply (\ref{3.24}).

\begin{proposition}
\label{Prop3.2}The sum $\sum_{n=0}^{q}L_{tsn}^{k}\left( x\right) $ can be
written as a linear combination of the $s$-Fibopolynomials $\binom{q+m}{tk}%
_{F_{s}\left( x\right) }$, $m=1,2,\ldots ,tk$, according to (\ref{3.23}), in
the following cases%
\begin{equation*}
\begin{tabular}{|c||c|c|c|}
\hline
& $t$ & $k$ & $s$ \\ \hline\hline
(a) & even & odd & even \\ \hline
(b) & $\equiv 0\func{mod}4$ & odd & any \\ \hline
\end{tabular}%
\end{equation*}
\end{proposition}

\begin{proof}
In both cases we have $tk$ even, so it is valid the factorization (\ref{3.9}%
) for $\Pi _{2}\left( x,z\right) $.

(a) Suppose that $t$ and $s$ are even and that $k$ is odd. In this case the
factor $\left( z-\left( -1\right) ^{\frac{kts}{2}}\right) $ in $\Pi
_{2}\left( x,z\right) $ is $\left( z-1\right) $, so we have $\Pi _{2}\left(
x,1\right) =0$. It remains to check that $\widetilde{\Pi }_{1}\left(
x,1\right) $ is finite. In fact, by writing $k$ as $2k-1$ and using that $t$
and $s$ are even, one can see that $\widetilde{\Pi }_{1}\left( x,1\right)
=4^{k-1}$.

(b) Suppose now that $t\equiv 0\func{mod}4$, $k$ is odd and $s$ is any
positive integer. In this case the factor $\left( z-\left( -1\right) ^{\frac{%
kts}{2}}\right) $ in $\Pi _{2}\left( x,z\right) $ is again $\left(
z-1\right) $, so we have $\Pi _{2}\left( x,1\right) =0$. With a similar
calculation to the case (a), we can see that in this case we have also $%
\widetilde{\Pi }_{1}\left( x,1\right) =4^{k-1}$.
\end{proof}

An example from the case (b) of proposition \ref{Prop3.2} is the following
identity (corresponding to $t=4$ and $k=1$), valid for any $s\in \mathbb{N}$%
\begin{eqnarray}
\sum_{n=0}^{q}L_{4sn}\left( x\right) &=&2\dbinom{q+4}{4}_{F_{s}\left(
x\right) }+\left( -L_{4s}\left( x\right) +2\left( -1\right)
^{s+1}L_{2s}\left( x\right) \right) \dbinom{q+3}{4}_{F_{s}\left( x\right) }
\label{3.31} \\
&&+\left( -1\right) ^{s}\left( L_{6s}\left( x\right) +L_{2s}\left( x\right)
+2\left( -1\right) ^{s}\right) \dbinom{q+2}{4}_{F_{s}\left( x\right)
}-L_{4s}\left( x\right) \dbinom{q+1}{4}_{F_{s}\left( x\right) }.  \notag
\end{eqnarray}

Examples from the cases (a) and (b) of proposition \ref{Prop3.1}, and from
the case (a) of proposition \ref{Prop3.2}, will be given in section \ref%
{Sec4}, since some variants of them work also as examples of alternating
sums of powers of Fibonacci or Lucas polynomials, to be discussed in section %
\ref{Sec4} (see corollaries \ref{Cor4.1} and \ref{Cor4.2}).

\section{\label{Sec4}The main results (II): alternating sums}

According to (\ref{2.3}), (\ref{2.5}), (\ref{2.7}), (\ref{2.71}) and (\ref%
{2.72}), the $Z$ transform of the alternating sequences $\left\{ \left(
-1\right) ^{n}F_{tsn}^{k}\left( x\right) \right\} _{n=0}^{\infty }$ and $%
\left\{ \left( -1\right) ^{n}L_{tsn}^{k}\left( x\right) \right\}
_{n=0}^{\infty }$ are%
\begin{eqnarray}
\mathcal{Z}\left( \left( -1\right) ^{n}F_{tsn}^{k}\left( x\right) \right) &=&%
\frac{1}{\left( x^{2}+4\right) ^{\frac{k}{2}}}\sum_{l=0}^{k}\binom{k}{l}%
\left( -1\right) ^{k-l}\frac{z}{z+\alpha ^{lts}\left( x\right) \beta
^{\left( k-l\right) ts}\left( x\right) }  \label{4.1} \\
&=&-z\frac{\sum_{i=0}^{tk}\sum_{j=0}^{i}\left( -1\right) ^{\frac{\left(
sj+2(s+1)\right) \left( j+1\right) }{2}}\dbinom{tk+1}{j}_{F_{s}\left(
x\right) }F_{ts\left( i-j\right) }^{k}\left( x\right) \left( -z\right)
^{tk-i}}{\sum_{i=0}^{tk+1}\left( -1\right) ^{\frac{\left( si+2(s+1)\right)
\left( i+1\right) }{2}}\dbinom{tk+1}{i}_{F_{s}\left( x\right) }\left(
-z\right) ^{tk+1-i}},  \notag
\end{eqnarray}%
and%
\begin{eqnarray}
\mathcal{Z}\left( \left( -1\right) ^{n}L_{tsn}^{k}\left( x\right) \right)
&=&\sum_{l=0}^{k}\binom{k}{l}\frac{z}{z+\alpha ^{lts}\left( x\right) \beta
^{\left( k-l\right) ts}\left( x\right) }  \label{4.2} \\
&=&-z\frac{\sum_{i=0}^{tk}\sum_{j=0}^{i}\left( -1\right) ^{\frac{\left(
sj+2(s+1)\right) \left( j+1\right) }{2}}\dbinom{tk+1}{j}_{F_{s}\left(
x\right) }L_{ts\left( i-j\right) }^{k}\left( x\right) \left( -z\right)
^{tk-i}}{\sum_{i=0}^{tk+1}\left( -1\right) ^{\frac{\left( si+2(s+1)\right)
\left( i+1\right) }{2}}\dbinom{tk+1}{i}_{F_{s}\left( x\right) }\left(
-z\right) ^{tk+1-i}}.  \notag
\end{eqnarray}

By using (\ref{2.81}) and (\ref{2.82}) it is possible to establish
expressions, for the case of alternating sums, similar to expressions (\ref%
{3.2}) and (\ref{3.22}) , namely%
\begin{eqnarray}
&&\sum_{n=0}^{q}\left( -1\right) ^{n}F_{tsn}^{k}\left( x\right)  \label{4.3}
\\
&=&\left( -1\right)
^{s+1}\sum_{m=1}^{tk}\sum_{i=0}^{tk-m}\sum_{j=0}^{i}\left( -1\right) ^{\frac{%
\left( sj+2(s+1)\right) \left( j+1\right) }{2}+i+tk+q+m}\dbinom{tk+1}{j}%
_{F_{s}\left( x\right) }F_{ts\left( i-j\right) }^{k}\left( x\right) \dbinom{%
q+m}{tk}_{F_{s}\left( x\right) }  \notag \\
&&+\left( -1\right) ^{s+1}\sum_{i=0}^{tk}\sum_{j=0}^{i}\left( -1\right) ^{%
\frac{\left( sj+2(s+1)\right) \left( j+1\right) }{2}+i+tk}\dbinom{tk+1}{j}%
_{F_{s}\left( x\right) }F_{ts\left( i-j\right) }^{k}\left( x\right)
\sum_{n=0}^{q}\left( -1\right) ^{n}\dbinom{n}{tk}_{F_{s}\left( x\right) }, 
\notag
\end{eqnarray}%
and%
\begin{eqnarray}
&&\sum_{n=0}^{q}\left( -1\right) ^{n}L_{tsn}^{k}\left( x\right)  \label{4.4}
\\
&=&\left( -1\right)
^{s+1}\sum_{m=1}^{tk}\sum_{i=0}^{tk-m}\sum_{j=0}^{i}\left( -1\right) ^{\frac{%
\left( sj+2(s+1)\right) \left( j+1\right) }{2}+i+tk+q+m}\dbinom{tk+1}{j}%
_{F_{s}\left( x\right) }L_{ts\left( i-j\right) }^{k}\left( x\right) \dbinom{%
q+m}{tk}_{F_{s}\left( x\right) }  \notag \\
&&+\left( -1\right) ^{s+1}\sum_{i=0}^{tk}\sum_{j=0}^{i}\left( -1\right) ^{%
\frac{\left( sj+2(s+1)\right) \left( j+1\right) }{2}+i+tk}\dbinom{tk+1}{j}%
_{F_{s}\left( x\right) }L_{ts\left( i-j\right) }^{k}\left( x\right)
\sum_{n=0}^{q}\left( -1\right) ^{n}\dbinom{n}{tk}_{F_{s}\left( x\right) }, 
\notag
\end{eqnarray}%
respectively. From (\ref{4.3}) and (\ref{4.4}) we see that the alternating
sums of powers $\sum_{n=0}^{q}\left( -1\right) ^{n}F_{tsn}^{k}\left(
x\right) $ and $\sum_{n=0}^{q}\left( -1\right) ^{n}L_{tsn}^{k}\left(
x\right) $ can be written as linear combinations of the $s$-Fibopolynomials $%
\binom{q+m}{tk}_{F_{s}\left( x\right) }$, $m=1,2,\ldots ,tk$, according to%
\begin{eqnarray}
&&\sum_{n=0}^{q}\left( -1\right) ^{n}F_{tsn}^{k}\left( x\right)  \label{4.5}
\\
&=&\left( -1\right)
^{s+1}\sum_{m=1}^{tk}\sum_{i=0}^{tk-m}\sum_{j=0}^{i}\left( -1\right) ^{\frac{%
\left( sj+2(s+1)\right) \left( j+1\right) }{2}+i+tk+q+m}\dbinom{tk+1}{j}%
_{F_{s}\left( x\right) }F_{ts\left( i-j\right) }^{k}\left( x\right) \dbinom{%
q+m}{tk}_{F_{s}\left( x\right) },  \notag
\end{eqnarray}%
and%
\begin{eqnarray}
&&\sum_{n=0}^{q}\left( -1\right) ^{n}L_{tsn}^{k}\left( x\right)  \label{4.6}
\\
&=&\left( -1\right)
^{s+1}\sum_{m=1}^{tk}\sum_{i=0}^{tk-m}\sum_{j=0}^{i}\left( -1\right) ^{\frac{%
\left( sj+2(s+1)\right) \left( j+1\right) }{2}+i+tk+q+m}\dbinom{tk+1}{j}%
_{F_{s}\left( x\right) }L_{ts\left( i-j\right) }^{k}\left( x\right) \dbinom{%
q+m}{tk}_{F_{s}\left( x\right) },  \notag
\end{eqnarray}%
if and only if we have that 
\begin{equation}
\sum_{i=0}^{tk}\sum_{j=0}^{i}\left( -1\right) ^{\frac{\left(
sj+2(s+1)\right) \left( j+1\right) }{2}+i}\dbinom{tk+1}{j}_{F_{s}\left(
x\right) }F_{ts\left( i-j\right) }^{k}\left( x\right) =0,  \label{4.7}
\end{equation}%
and%
\begin{equation}
\sum_{i=0}^{tk}\sum_{j=0}^{i}\left( -1\right) ^{\frac{\left(
sj+2(s+1)\right) \left( j+1\right) }{2}+i}\dbinom{tk+1}{j}_{F_{s}\left(
x\right) }L_{ts\left( i-j\right) }^{k}\left( x\right) =0,  \label{4.8}
\end{equation}%
respectively. Observe that, according to (\ref{4.1}) and (\ref{4.2}), we
have 
\begin{eqnarray}
&&\left( x^{2}+4\right) ^{\frac{k}{2}}\sum_{i=0}^{tk}\sum_{j=0}^{i}\left(
-1\right) ^{\frac{\left( sj+2(s+1)\right) \left( j+1\right) }{2}+i}\dbinom{%
tk+1}{j}_{F_{s}\left( x\right) }F_{ts\left( i-j\right) }^{k}\left( x\right)
z^{tk-i}  \label{4.9} \\
&=&\left( \sum_{l=0}^{k}\binom{k}{l}\left( -1\right) ^{k-l}\frac{1}{z+\alpha
^{lts}\left( x\right) \beta ^{\left( k-l\right) ts}\left( x\right) }\right)
\left( \sum_{i=0}^{tk+1}\left( -1\right) ^{\frac{\left( si+2(s+1)\right)
\left( i+1\right) }{2}+i}\dbinom{tk+1}{i}_{F_{s}\left( x\right)
}z^{tk+1-i}\right) ,  \notag
\end{eqnarray}%
and%
\begin{eqnarray}
&&\sum\limits_{i=0}^{tk}\!\sum\limits_{j=0}^{i}\left( -1\right) ^{\frac{%
\left( sj+2(s+1)\right) \left( j+1\right) }{2}+i}\!\dbinom{tk+1}{j}%
_{F_{s}\left( x\right) }\!L_{ts\left( i-j\right) }^{k}\!\left( x\right)
z^{tk-i}  \label{4.10} \\
&=&\left( \sum_{l=0}^{k}\binom{k}{l}\frac{1}{z+\alpha ^{lts}\left( x\right)
\beta ^{\left( k-l\right) ts}\left( x\right) }\right) \left(
\sum\limits_{i=0}^{tk+1}\left( -1\right) ^{\frac{\left( si+2(s+1)\right)
\left( i+1\right) }{2}+i}\dbinom{tk+1}{i}_{F_{s}\left( x\right)
}z^{tk+1-i}\right) ,  \notag
\end{eqnarray}%
respectively. Then we need to consider now the following factors%
\begin{equation}
\Omega _{1}\left( x,z\right) =\sum_{l=0}^{k}\binom{k}{l}\left( -1\right)
^{k-l}\frac{1}{z+\alpha ^{lts}\left( x\right) \beta ^{\left( k-l\right)
ts}\left( x\right) },  \label{4.11}
\end{equation}

\begin{equation}
\widetilde{\Omega }_{1}\left( x,z\right) =\sum_{l=0}^{k}\binom{k}{l}\frac{1}{%
z+\alpha ^{lts}\left( x\right) \beta ^{\left( k-l\right) ts}\left( x\right) }%
,  \label{4.12}
\end{equation}

and 
\begin{equation}
\Omega _{2}\left( x,z\right) =\sum_{i=0}^{tk+1}\left( -1\right) ^{\frac{%
\left( si+2(s+1)\right) \left( i+1\right) }{2}+i}\dbinom{tk+1}{i}%
_{F_{s}\left( x\right) }z^{tk+1-i}.  \label{4.13}
\end{equation}

Plainly (\ref{4.7}) is concluded from any of the conditions: (a) $\Omega
_{1}\left( x,1\right) =0$, or, (b) $\Omega _{1}\left( x,1\right) <\infty $
and $\Omega _{2}\left( x,1\right) =0$, and (\ref{4.8}) is concluded from any
of the conditions: (a) $\widetilde{\Omega }_{1}\left( x,1\right) =0$, or,
(b) $\widetilde{\Omega }_{1}\left( x,1\right) <\infty $ and $\Omega
_{2}\left( x,1\right) =0$. In this section we give conditions on the
parameters $t$, $k$ and $s$, that imply (\ref{4.7}) (for the Fibonacci case:
proposition \ref{Prop4.1}), and that imply (\ref{4.8}) (for the Lucas case:
proposition \ref{Prop4.2}).

In the Fibonacci case we have the following result.

\begin{proposition}
\label{Prop4.1}The alternating sum $\sum_{n=0}^{q}\left( -1\right)
^{n}F_{tsn}^{k}\left( x\right) $ can be written as a linear combination of
the $s$-Fibopolynomials $\binom{q+m}{tk}_{F_{s}\left( x\right) }$, $%
m=1,2,\ldots ,tk$, according to (\ref{4.5}), in the following cases%
\begin{equation*}
\begin{tabular}{|c||c|c|c|}
\hline
& $t$ & $k$ & $s$ \\ \hline\hline
(a) & any & $\equiv 0\func{mod}4$ & any \\ \hline
(b) & any & even & even \\ \hline
(c) & $\equiv 2\func{mod}4$ & any & odd \\ \hline
(d) & even & even & any \\ \hline
\end{tabular}%
\end{equation*}
\end{proposition}

\begin{proof}
Observe that in all the four cases we have $tk$ even. Thus, according to (%
\ref{2.11}) (with $z$ replaced by $-z$) we can factor $\Omega _{2}\left(
x,z\right) $ as%
\begin{equation}
\Omega _{2}\left( x,z\right) =\left( -1\right) ^{s}\left( z+\left( -1\right)
^{\frac{tsk}{2}}\right) \prod_{j=0}^{\frac{tk}{2}-1}\left( z^{2}+\left(
-1\right) ^{sj}L_{2s\left( \frac{tk}{2}-j\right) }\left( x\right) z+1\right)
.  \label{4.14}
\end{equation}

(a) Suppose that $k\equiv 0\func{mod}4$ and that $t$ and $s$ are any
positive integers. In this case the factor $\left( z+\left( -1\right) ^{%
\frac{tsk}{2}}\right) $ of (\ref{4.14}) is $z+1$, so we have $\Omega
_{2}\left( x,1\right) \neq 0$. However, by setting $z=1$ in (\ref{4.11}),
with $k$ replaced by $4k$, we get%
\begin{equation*}
\Omega _{1}\left( x,1\right) =\sum_{l=0}^{2k-1}\binom{4k}{l}\left( -1\right)
^{l}+\frac{1}{2}\binom{4k}{2k}=0.
\end{equation*}

Thus (\ref{4.7}) holds, as wanted.

(b) Suppose that $k$ and $s$ are even, and that $t$ is any positive integer.
In this case we have $z+\left( -1\right) ^{\frac{tsk}{2}}=z+1$, and then $%
\Omega _{2}\left( x,1\right) \neq 0$. By setting $z=1$ in (\ref{4.11}), with 
$k$ and $s$ substituted by $2k$ and $2s$, respectively, we get%
\begin{equation*}
\Omega _{1}\left( x,1\right) =\sum_{l=0}^{2k}\binom{2k}{l}\left( -1\right)
^{l}\frac{1}{1+\alpha ^{2lts}\left( x\right) \beta ^{\left( 2k-l\right)
2ts}\left( x\right) }=\sum_{l=0}^{k-1}\binom{2k}{l}\left( -1\right) ^{l}+%
\frac{1}{2}\binom{2k}{k}\left( -1\right) ^{k}=0.
\end{equation*}

Thus (\ref{4.7}) holds, as wanted.

(c) Suppose that $s$ is odd, $t\equiv 2\func{mod}4$, and $k$ is any positive
integer. If in (\ref{4.14}) we set $z=1$ and replace $t$ by $2\left(
2t-1\right) $, we obtain that%
\begin{equation}
\Omega _{2}\left( x,1\right) =\left( 1+\left( -1\right) ^{k}\right)
\prod_{j=0}^{\left( 2t-1\right) k-1}\left( \left( -1\right) ^{j}L_{2s\left(
\left( 2t-1\right) k-j\right) }\left( x\right) +2\right) .  \label{4.15}
\end{equation}

We consider two subcases:

(c1) Suppose that $k$ is even. In this case we have $\Omega _{2}\left(
x,1\right) \neq 0$. But if we set $z=1$ in (\ref{4.11}), replace $k$ by $2k$%
, and use that $t\equiv 2\func{mod}4$ and that $s$ is odd, we obtain that%
\begin{equation*}
\Omega _{1}\left( x,1\right) =\sum_{l=0}^{k-1}\binom{2k}{l}\left( -1\right)
^{l}+\frac{1}{2}\binom{2k}{k}\left( -1\right) ^{k}=0.
\end{equation*}

Thus (\ref{4.7}) holds when $k$ is even.

(c2) Suppose that $k$ is odd. In this case we have clearly that $\Omega
_{2}\left( x,1\right) =0$. We check that $\Omega _{1}\left( x,1\right) $ is
finite. If we set $z=1$ in (\ref{4.11}), substitute $k$ by $2k-1$, and use
that $t\equiv 2\func{mod}4$ and that $s$ is odd, we obtain that 
\begin{equation*}
\Omega _{1}\left( x,1\right) =\sqrt{x^{2}+4}\sum_{l=0}^{k-1}\binom{2k-1}{l}%
\left( -1\right) ^{l+1}\frac{F_{\left( 2k-1-2l\right) ts}\left( x\right) }{%
2+L_{\left( 2k-1-2l\right) ts}\left( x\right) }.
\end{equation*}

Then we have $\Omega _{1}\left( x,1\right) <\infty $, as wanted. That is,
expression (\ref{4.7}) holds when $k$ is odd.

(d) Suppose that $k$ and $t$ are even and $s$ is any positive integer. In
this case the factor $\left( z+\left( -1\right) ^{\frac{tsk}{2}}\right) $ of
(\ref{4.14}) is $\left( z+1\right) $, so we have $\Omega _{2}\left(
x,1\right) \neq 0$. Observe that, replacing \thinspace $k$ and $t$ by $2k$
and $2t$, respectively (and letting $s$ be any natural number) we obtain the
same expression for $\Omega _{1}\left( x,1\right) $ of the case (b), namely%
\begin{equation*}
\Omega _{1}\left( x,1\right) =\sum_{l=0}^{2k}\binom{2k}{l}\left( -1\right)
^{l}\frac{1}{1+\alpha ^{2lts}\left( x\right) \beta ^{\left( 2k-l\right)
2ts}\left( x\right) },
\end{equation*}%
which is equal to $0$. That is, in this case we have also that $\Omega
_{1}\left( x,1\right) =0$, and we conclude that (\ref{4.7}) holds.
\end{proof}

\begin{corollary}
\label{Cor4.1}

(a) If $t$ is odd and $k\equiv 2\func{mod}4$, we have the following identity
valid for any $s\in \mathbb{N}$ 
\begin{eqnarray}
&&\sum_{n=0}^{q}\left( -1\right) ^{\left( s+1\right) n}F_{tsn}^{k}\left(
x\right)  \label{4.19} \\
&=&\left( -1\right)
^{s+1}\sum_{m=1}^{tk}\sum_{i=0}^{tk-m}\sum_{j=0}^{i}\left( -1\right) ^{\frac{%
\left( sj+2(s+1)\right) \left( j+1\right) }{2}+\left( s+1\right) \left(
i+tk+q+m\right) }\dbinom{tk+1}{j}_{F_{s}\left( x\right) }F_{ts\left(
i-j\right) }^{k}\left( x\right) \dbinom{q+m}{tk}_{F_{s}\left( x\right) }. 
\notag
\end{eqnarray}

(b) If $t\equiv 2\func{mod}4$ and $k$ is odd, we have the following identity
valid for any $s\in \mathbb{N}$ 
\begin{eqnarray}
&&\sum_{n=0}^{q}\left( -1\right) ^{sn}F_{tsn}^{k}\left( x\right)
\label{4.20} \\
&=&\left( -1\right)
^{s+1}\sum_{m=1}^{tk}\sum_{i=0}^{tk-m}\sum_{j=0}^{i}\left( -1\right) ^{\frac{%
\left( sj+2(s+1)\right) \left( j+1\right) }{2}+s\left( i+tk+q+m\right) }%
\dbinom{tk+1}{j}_{F_{s}\left( x\right) }F_{ts\left( i-j\right) }^{k}\left(
x\right) \dbinom{q+m}{tk}_{F_{s}\left( x\right) }.  \notag
\end{eqnarray}
\end{corollary}

\begin{proof}
(a) When $t$ is odd and $k\equiv 2\func{mod}4$, formula (\ref{4.19}) with $s$
odd gives the result (\ref{3.3}) of case (b) of proposition \ref{Prop3.1}
(which is valid for $t$ odd, $k\equiv 2\func{mod}4$ and $s$ odd). Similarly,
for $t$ odd and $k\equiv 2\func{mod}4$, formula (\ref{4.19}) with $s$ even,
gives the result (\ref{4.5}) of case (b) of proposition \ref{Prop4.1} (which
is valid for $k$ and $s$ even and any $t$).

(b) When $t\equiv 2\func{mod}4$ and $k$ is odd, formula (\ref{4.20}) with $s$
even, gives the result (\ref{3.3}) of case (a) of proposition \ref{Prop3.1}
(which is valid for $t$ even, $k$ odd and $s$ even). Similarly, for $t\equiv
2\func{mod}4$ and $k$ odd, formula (\ref{4.20}) with $s$ odd, gives the
result (\ref{4.5}) of case (c) of proposition \ref{Prop4.1} (which is valid
for $t\equiv 2\func{mod}4$, $s$ odd and any $k$).
\end{proof}

We give examples from the cases considered in corollary \ref{Cor4.1}.
Beginning with the case (a), by setting $t=1$ and $k=2$ in (\ref{4.19}), we
have the following identity, valid for $s\in \mathbb{N}$

\begin{equation}
\sum_{n=0}^{q}\left( -1\right) ^{\left( s+1\right) \left( n+q\right)
}F_{sn}^{2}\left( x\right) =F_{s}^{2}\left( x\right) \dbinom{q+1}{2}%
_{F_{s}\left( x\right) }.  \label{4.21}
\end{equation}

The case $t=s=x=1$ and $k=6$ of (\ref{4.19}) is 
\begin{equation*}
\sum_{n=0}^{q}F_{n}^{6}=\dbinom{q+1}{6}_{F}+\dbinom{q+5}{6}_{F}-11\left( 
\dbinom{q+2}{6}_{F}+\dbinom{q+4}{6}_{F}\right) -64\dbinom{q+3}{6}_{F}.
\end{equation*}

This identity is mentioned in \cite{Pi1} (p. 259), and previously was
obtained in \cite{O-N} with the much more simple right-hand side $\frac{1}{4}%
\left( F_{q}^{5}F_{q+3}+F_{2q}\right) $.

An example from the case (b) of corollary \ref{Cor4.1} is the following
identity, valid for $s\in \mathbb{N}$ (obtained by setting $t=2$ and $k=1$
in (\ref{4.20}))%
\begin{equation}
\sum_{n=0}^{q}\left( -1\right) ^{s\left( n+q\right) }F_{2sn}\left( x\right)
=F_{2s}\left( x\right) \dbinom{q+1}{2}_{F_{s}\left( x\right) }.  \label{4.24}
\end{equation}

From (\ref{4.21}) and (\ref{4.24}) we see that%
\begin{equation}
\left( -1\right) ^{\left( s+1\right) q}L_{s}\left( x\right)
\sum_{n=0}^{q}\left( -1\right) ^{\left( s+1\right) n}F_{sn}^{2}\left(
x\right) =\left( -1\right) ^{sq}F_{s}\left( x\right) \sum_{n=0}^{q}\left(
-1\right) ^{sn}F_{2sn}\left( x\right) =F_{s\left( q+1\right) }\left(
x\right) F_{sq}\left( x\right) .  \label{4.27}
\end{equation}

The simplest example from the case (a) of proposition \ref{Prop4.1},
corresponding to $k=4$ and $t=1$, is the following identity valid for any $%
s\in \mathbb{N}$%
\begin{equation}
\sum_{n=0}^{q}\left( -1\right) ^{n+q}F_{sn}^{4}\left( x\right)
=F_{s}^{4}\left( x\right) \left( \dbinom{q+1}{4}_{F_{s}\left( x\right) }+%
\dbinom{q+3}{4}_{F_{s}\left( x\right) }+\left( 3\left( -1\right)
^{s}L_{2s}\left( x\right) +4\right) \dbinom{q+2}{4}_{F_{s}\left( x\right)
}\right) .  \label{4.36}
\end{equation}

With some patience one can see that the case $x=1$ of (\ref{4.36}) is 
\begin{equation*}
\sum_{n=0}^{q}\left( -1\right) ^{n}F_{sn}^{4}=\frac{\left( -1\right)
^{q}F_{sq}F_{s\left( q+1\right) }\left( L_{s}L_{sq}L_{s\left( q+1\right)
}-4L_{2s}\right) }{5L_{s}L_{2s}},
\end{equation*}%
demonstrated by Melham \cite{Mel2}.

An example from the case (d) of proposition \ref{Prop4.1}, corresponding to $%
t=k=2$, is the following identity valid for any $s\in \mathbb{N}$%
\begin{equation}
\sum_{n=0}^{q}\left( -1\right) ^{n+q}F_{2sn}^{2}\left( x\right)
=F_{2s}^{2}\left( x\right) \left( \dbinom{q+1}{4}_{F_{s}\left( x\right)
}+\left( -1\right) ^{s+1}L_{2s}\left( x\right) \dbinom{q+2}{4}_{F_{s}\left(
x\right) }+\dbinom{q+3}{4}_{F_{s}\left( x\right) }\right) .  \label{4.37}
\end{equation}

Now we consider alternating sums of powers of Lucas polynomials.

\begin{proposition}
\label{Prop4.2}The alternating sum $\sum_{n=0}^{q}\left( -1\right)
^{n}L_{tsn}^{k}\left( x\right) $ can be written as a linear combination of
the $s$-Fibopolynomials $\binom{q+m}{tk}_{F_{s}\left( x\right) }$, $%
m=1,2,\ldots ,tk$, according to (\ref{4.6}), if $s$ and $k$ are odd positive
integers and $t\equiv 2\func{mod}4.$
\end{proposition}

\begin{proof}
We will show that in the case stated in the proposition we have $\widetilde{%
\Omega }_{1}\left( x,1\right) <\infty $ and $\Omega _{2}\left( x,1\right) =0$%
, which implies (\ref{4.8}). Since $s$ and $k$ are odd, and $t\equiv 2\func{%
mod}4$, the factor $\left( z+\left( -1\right) ^{\frac{tsk}{2}}\right) $ of (%
\ref{4.14}) is $\left( z-1\right) $, so we have $\Omega _{2}\left(
x,1\right) =0$. Let us see that $\widetilde{\Omega }_{1}\left( x,1\right)
<\infty $. If in (\ref{4.12}) we set $z=1$ and replace $k$ by $2k-1$, we get
for $t\equiv 2\func{mod}4$ and $s$ odd that $\widetilde{\Omega }_{1}\left(
x,1\right) =4^{k-1}$, as wanted.
\end{proof}

\begin{corollary}
\label{Cor4.2}If $t\equiv 2\func{mod}4$ and $k$ is odd, we have the
following identity valid for any $s\in \mathbb{N}$ 
\begin{eqnarray}
&&\sum_{n=0}^{q}\left( -1\right) ^{sn}L_{tsn}^{k}\left( x\right)
\label{4.43} \\
&=&\left( -1\right)
^{s+1}\sum_{m=1}^{tk}\sum_{i=0}^{tk-m}\sum_{j=0}^{i}\left( -1\right) ^{\frac{%
\left( sj+2(s+1)\right) \left( j+1\right) }{2}+s\left( i+tk+q+m\right) }%
\dbinom{tk+1}{j}_{F_{s}\left( x\right) }L_{ts\left( i-j\right) }^{k}\left(
x\right) \dbinom{q+m}{tk}_{F_{s}\left( x\right) }.  \notag
\end{eqnarray}
\end{corollary}

\begin{proof}
When $t\equiv 2\func{mod}4$ and $k$ is odd, formula (\ref{4.43}) with $s$
even, gives the result (\ref{3.23}) of case (a) of proposition \ref{Prop3.2}
(which is valid for $t$ even, $k$ odd and $s$ even). Similarly, if $t\equiv 2%
\func{mod}4$ and $k$ is odd, formula (\ref{4.43}) with $s$ odd, gives the
result (\ref{4.6}) of proposition \ref{Prop4.2} (which is valid for $t\equiv
2\func{mod}4$, $k$ odd and $s$ odd).
\end{proof}

An example of (\ref{4.43}) is the following identity (corresponding to $t=2$
and $k=1$), valid for any $s\in \mathbb{N}$ 
\begin{equation}
\sum_{n=0}^{q}\left( -1\right) ^{s\left( n+q\right) }L_{2sn}\left( x\right)
=2\dbinom{q+2}{2}_{F_{s}\left( x\right) }-L_{2s}\left( x\right) \dbinom{q+1}{%
2}_{F_{s}\left( x\right) },  \label{4.44}
\end{equation}%
which can be written as%
\begin{equation}
\sum_{n=0}^{q}\left( -1\right) ^{s\left( n+q\right) }L_{2sn}\left( x\right) =%
\frac{1}{F_{s}\left( x\right) }L_{sq}\left( x\right) F_{s\left( q+1\right)
}\left( x\right) .  \label{4.45}
\end{equation}

\section{\label{Sec5}Further results}

To end this work we want to present (in two propositions) some examples of
identities obtained as derivatives of some of our previous results. We will
use the identities%
\begin{eqnarray}
\frac{d}{dx}F_{n}\left( x\right) &=&\frac{1}{x^{2}+4}\left( nL_{n}\left(
x\right) -xF_{n}\left( x\right) \right) .  \label{5.1} \\
\frac{d}{dx}L_{n}\left( x\right) &=&nF_{n}\left( x\right) .  \label{5.2}
\end{eqnarray}

One can see that these formulas are true by checking that both sides of each
one have the same $Z$ transform. By using (\ref{2.31}) and (\ref{2.75}) we
see that the $Z$ transform of both sides of (\ref{5.1}) is $z^{2}\left(
z^{2}-xz-1\right) ^{-2}$, and that the $Z$ transform of both sides of (\ref%
{5.2}) is $z\left( z^{2}+1\right) \left( z^{2}-xz-1\right) ^{-2}$.

\begin{proposition}
The following identities hold%
\begin{equation}
\left( -1\right) ^{\left( \!s+1\!\right) q}\!2L_{s}\!\left( x\right)
\!\sum_{n=0}^{q}\!\left( -1\right) ^{\left( \!s+1\!\right)
n}nF_{2sn}\!\left( x\right) =2qF_{s\left( 2q+1\right) }\!\left( x\right)
+F_{sq}\!\left( x\right) \!L_{s\left( q+1\right) }\!\left( x\right) -\frac{%
\left( \!x^{2}\!+\!4\!\right) \!F_{s}\!\left( x\right) }{L_{s}\!\left(
x\right) }F_{s\left( q+1\right) }\!\left( x\right) \!F_{sq}\!\left( x\right)
.  \label{5.4}
\end{equation}%
\begin{equation}
\left( -1\right) ^{sq}\!2F_{s}\!\left( x\right) \!\sum_{n=0}^{q}\!\left(
-1\right) ^{sn}nL_{2sn}\!\left( x\right) =2qF_{s\left( 2q+1\right) }\!\left(
x\right) +F_{sq}\!\left( x\right) \!L_{s\left( q+1\right) }\!\left( x\right)
-\frac{L_{s}\!\left( x\right) }{F_{s}\!\left( x\right) }F_{s\left(
q+1\right) }\!\left( x\right) \!F_{sq}\!\left( x\right) .  \label{5.5}
\end{equation}
\end{proposition}

\begin{proof}
We will use (\ref{4.27}), which contains two identities, namely 
\begin{equation}
\left( -1\right) ^{\left( s+1\right) q}L_{s}\left( x\right)
\sum_{n=0}^{q}\left( -1\right) ^{\left( s+1\right) n}F_{sn}^{2}\left(
x\right) =F_{s\left( q+1\right) }\left( x\right) F_{sq}\left( x\right) ,
\label{5.37}
\end{equation}%
and%
\begin{equation}
\left( -1\right) ^{sq}F_{s}\left( x\right) \sum_{n=0}^{q}\left( -1\right)
^{sn}F_{2sn}\left( x\right) =F_{s\left( q+1\right) }\left( x\right)
F_{sq}\left( x\right) .  \label{5.38}
\end{equation}

By using that $L_{sq}\left( x\right) F_{s\left( q+1\right) }\left( x\right)
+F_{sq}\left( x\right) L_{s\left( q+1\right) }\left( x\right) =2F_{s\left(
2q+1\right) }\left( x\right) $ (see (\ref{1.101})), we can see that 
\begin{equation*}
\frac{d}{dx}\left( F_{s\left( q+1\right) }\left( x\right) F_{sq}\left(
x\right) \right) =\frac{1}{x^{2}+4}\left( 2sqF_{s\left( 2q+1\right) }\left(
x\right) +sF_{sq}\left( x\right) L_{s\left( q+1\right) }\left( x\right)
-2xF_{s\left( q+1\right) }\left( x\right) F_{sq}\left( x\right) \right) .
\end{equation*}

The derivative of the left-hand side of (\ref{5.37}) is%
\begin{eqnarray*}
&&\frac{d}{dx}\left( \left( -1\right) ^{\left( s+1\right) q}L_{s}\left(
x\right) \sum_{n=0}^{q}\left( -1\right) ^{\left( s+1\right)
n}F_{sn}^{2}\left( x\right) \right) \\
&=&\left( -1\right) ^{\left( s+1\right) q}L_{s}\left( x\right) \sum_{n=0}^{q}%
\frac{\left( -1\right) ^{\left( s+1\right) n}}{x^{2}+4}2snF_{2sn}\left(
x\right) +\left( sF_{s}\left( x\right) -\frac{2xL_{s}\left( x\right) }{%
x^{2}+4}\right) \frac{1}{L_{s}\left( x\right) }F_{s\left( q+1\right) }\left(
x\right) F_{sq}\left( x\right) .
\end{eqnarray*}

Then, the derivative of (\ref{5.37}) is%
\begin{eqnarray*}
&&\left( -1\right) ^{\left( s+1\right) q}L_{s}\left( x\right) \sum_{n=0}^{q}%
\frac{\left( -1\right) ^{\left( s+1\right) n}}{x^{2}+4}2snF_{2sn}\left(
x\right) +\left( sF_{s}\left( x\right) -\frac{2xL_{s}\left( x\right) }{%
x^{2}+4}\right) \frac{1}{L_{s}\left( x\right) }F_{s\left( q+1\right) }\left(
x\right) F_{sq}\left( x\right) \\
&=&\frac{1}{x^{2}+4}\left( 2sqF_{s\left( 2q+1\right) }\left( x\right)
+sF_{sq}\left( x\right) L_{s\left( q+1\right) }\left( x\right)
-2xF_{sq}\left( x\right) F_{s\left( q+1\right) }\left( x\right) \right) ,
\end{eqnarray*}%
from where (\ref{5.4}) follows.

The derivative of the left-hand side of (\ref{5.38}) is%
\begin{eqnarray*}
&&\frac{d}{dx}\left( \left( -1\right) ^{sq}F_{s}\left( x\right)
\sum_{n=0}^{q}\left( -1\right) ^{sn}F_{2sn}\left( x\right) \right) \\
&=&\frac{\left( -1\right) ^{sq}}{x^{2}+4}F_{s}\left( x\right)
\sum_{n=0}^{q}\left( -1\right) ^{sn}2snL_{2sn}+\frac{sL_{s}\left( x\right)
-2xF_{s}\left( x\right) }{F_{s}\left( x\right) \left( x^{2}+4\right) }%
F_{s\left( q+1\right) }\left( x\right) F_{sq}\left( x\right) .
\end{eqnarray*}

Thus, the derivative of (\ref{5.38}) is%
\begin{eqnarray*}
&&\frac{\left( -1\right) ^{sq}}{x^{2}+4}F_{s}\left( x\right)
\sum_{n=0}^{q}\left( -1\right) ^{sn}2snL_{2sn}+\frac{sL_{s}\left( x\right)
-2xF_{s}\left( x\right) }{F_{s}\left( x\right) \left( x^{2}+4\right) }%
F_{s\left( q+1\right) }\left( x\right) F_{sq}\left( x\right) \\
&=&\frac{1}{x^{2}+4}\left( 2sqF_{s\left( 2q+1\right) }\left( x\right)
+sF_{sq}\left( x\right) L_{s\left( q+1\right) }\left( x\right) -2xF_{s\left(
q+1\right) }\left( x\right) F_{sq}\left( x\right) \right) ,
\end{eqnarray*}%
from where (\ref{5.5}) follows.
\end{proof}

\begin{proposition}
The following identity holds%
\begin{equation}
\left( -1\right) ^{sq}F_{s}\left( x\right) \sum_{n=0}^{q}\left( -1\right)
^{sn}nF_{2sn}\left( x\right) =\frac{1}{x^{2}+4}\left( \frac{\left( -1\right)
^{s+1}}{F_{s}\left( x\right) }F_{2sq}\left( x\right) +qL_{s\left(
2q+1\right) }\left( x\right) \right) .  \label{5.39}
\end{equation}
\end{proposition}

\begin{proof}
Identity (\ref{5.39}) is the derivative of (\ref{4.45}), together with 
\begin{equation*}
F_{s}\left( x\right) L_{s\left( q+1\right) }\left( x\right) -L_{s}\left(
x\right) F_{s\left( q+1\right) }\left( x\right) =2\left( -1\right)
^{s+1}F_{sq}\left( x\right) ,
\end{equation*}%
and%
\begin{equation*}
\left( x^{2}+4\right) F_{sq}\left( x\right) F_{s\left( q+1\right) }\left(
x\right) +L_{sq}\left( x\right) L_{s\left( q+1\right) }\left( x\right)
=2L_{s\left( 2q+1\right) }\left( x\right) .
\end{equation*}

(See (\ref{1.101}) and (\ref{1.102}).) We leave the details of the
calculations to the reader.
\end{proof}

\section{Acknowledgments}

I thank the anonymous referee for the careful reading of the first version
of this article, and the valuable comments and suggestions.

\end{document}